%% file: LPAR2017CERLET.tex
\documentclass{easychair}

\usepackage{doc}
\usepackage[utf8x]{inputenc}
\usepackage{amsmath}
\usepackage{footmisc}
\usepackage{stmaryrd}
\usepackage{amsthm}
\usepackage{amsfonts}
\usepackage{color}
\usepackage{mathtools}
\usepackage{algpseudocode}
\usepackage{algorithm}
\usepackage{mathabx}
\usepackage{bussproofs}
\usepackage{rotating}
\usepackage{relsize}
\usepackage{tikz}
\usepackage{amssymb}
\usetikzlibrary{arrows,positioning}
\usetikzlibrary{backgrounds}
\makeatletter

\title{Clausal Analysis of First-order Proof Schemata}

\author{
David M. Cerna\inst{1}
\and
Michael Lettmann\inst{2}\thanks{Funded by FWF project W1255-N23.}
}

\institute{
  Reseach Institute for Symbolic Computation\\
Johannes Kepler University\\
Linz, Austria\\
\email{david.cerna@risc.jku.at}
\and
   Institute of Information Systems\\
   Technische Universit\"at Wien\\
    Vienna, Austria\\
    \email{lettmann@logic.at}
 }

\authorrunning{Cerna and Lettmann}
\titlerunning{Clausal Analysis of Proof Schemata}
\include{Symbols}
\begin{document}

\maketitle
\begin{abstract}
Proof schemata are a variant of \LK-proofs able to simulate various induction schemes in first-order logic by adding so called proof links to the standard first-order \LK-calculus. Proof links allow proofs to reference proofs thus giving proof schemata a recursive structure. Unfortunately, applying {\em reductive} cut-elimination is non-trivial in the presence of proof links. Borrowing the concept of lazy instantiation from functional programming, 
we evaluate proof links locally allowing reductive cut-elimination to proceed past them. Though, this method cannot be used to obtain cut-free proof schemata, we nonetheless obtain important results concerning the schematic
\CERES\ method, that is a method of cut-elimination for proof schemata based on resolution. 
In ``Towards a clausal analysis of cut-elimination'', it was shown that reductive cut-elimination transforms a given \LK-proof in such a way that a subsumption relation holds between the pre- and post-transformation {\em characteristic clause sets}, i.e. the clause set representing the cut-structure of an \LK-proof. Let $\mathit{CL}(\varphi ')$ be the characteristic clause set of a normal form $\varphi '$ of an \LK-proof $\varphi$ that is reached by performing reductive cut-elimination on $\varphi$ without atomic cut elimination. Then $\mathit{CL}(\varphi ')$ is subsumed by all characteristic clause sets extractable from any application of reductive cut-elimination to $\varphi$. Such a normal form is referred to as an \ACNFtop\ and plays an essential role in methods of cut-elimination by resolution. These results can be extended to proof schemata through our ``lazy instantiation'' of proof links, and provides an essential step toward a complete cut-elimination method for proof schemata. 
\end{abstract}

%------------------------------------------------------------------------------
\section{Introduction}
The {\em schematic} \CERES\ (\textbf{C}ut \textbf{E}limination by 
\textbf{RES}olution) method of cut-elimination 
was developed for a primitive recursively defined first-order sequent calculus, the 
\LKS-calculus~\cite{CERESS}. Note that \LKS-proofs have a {\em free parameter}, which 
when instantiated results in an \LK-proof. The method is based on \CERES\ which is 
an alternative to the reductive cut-elimination method of Gentzen~\cite{Gentzen1935}. 
It relies on the extraction of a {\em characteristic clause set} $C$ from a sequent 
calculus proof $\varphi$, where $C$ represents the global cut structure of $\varphi$. 
$C$ is always unsatisfiable and thus, refutable using resolution. A refutation $R$ 
of $C$ can be used to construct a proof $\varphi'$ ending with the same sequent as 
$\varphi$, $\ES{\varphi}$, but containing only atomic cuts. Note that 
removal of atomic cuts is computationally inexpensive.  The formulas of 
$\ES{\varphi}$ can be proven by {\em proof projections}, which are small cut-free proofs algorithmically extracted  from $\varphi$. Note that proof projections also contain ancestors of the cut formulas. That is the end sequents of the proof 
projections are of the form $\ES{\varphi}\cup c$ for some $c\in C$. 
Attaching the projections to the leaves of $R$ and adding the necessary contractions 
results in $\varphi'$. Benefits of the \CERES\ method are its 
non-elementary speed-up over the Gentzen method~\cite{Baaz:2013:MC:2509679} and its ``global'' approach to the elimination of cuts. A global approach is essential for analysis of proofs with an inductive argument.

The \LKS-calculus is based on {\em propositional schemata} introduced by Aravantinos 
\textit{et al.}~\cite{Aravantinos2009,Aravantinos:2011:DUR:2016945.2016961}. 
Propositional schemata have been used in the field of inductive theorem 
proving~\cite{Aravantinos:2013:RCF:2594934.2594935} to define more expressive 
classes of inductive formula whose satisfiability is decidable. The proofs of the 
\LKS-calculus are used to construct {\em proof schemata} which are denoted using a 
finite ordered list of {\em proof schema components}. The components consist of a 
{\em proof symbol}, an \LKS-proof and an \LKE-proof. The \LKE-calculus is essentially 
the \LK-calculus with a rudimentary equational rule for dealing with recursively 
defined functions and predicates. In addition, \LKS-proofs are allowed to ``call'' other proofs through {\em proof links}. To maintain well foundedness, a constraint was added to proof schemata that an \LKS-proof can only contain proof links to proof schema components lower in the ordering. Also, no ``free'' proof links are allowed, 
i.e.\ proofs links must point to proofs within its proof schema. It was shown that 
proof schemata, under these constraints, can be used to express certain inductive 
arguments~\cite{CERESS}. One way to think about proof schemata is as a countable set of 
\LK-proofs. Thus, cut-elimination would be the elimination of cuts from a countable set of proofs 
simultaneously. One can imagine why such a problem would benefit from a global approach. 

Induction has been a stumbling block for classical cut-elimination methods, i.e.\ 
reductive 
cut-elimination \`{a} la Gentzen, because passing cuts over the inference rule 
for induction can be unsound. Note that this is still an issue for the  more exotic 
(non-reductive) methods such as \CERES\ because the cut ancestor relation is not 
well defined~\cite{CERESS}. Though, there have been systems developed to get around 
the issue, such as the work of Mcdowell and 
Miller~\cite{Mcdowell97cut-eliminationfor}, and Brotherston and 
Simpson~\cite{4276551}. These systems manage to eliminate cuts but without some
of the benefits of cut-elimination like the 
{\em subformula property}, i.e.  every formula occurring in the derivation is a 
subformula of some formula in the end sequent. This property admits the construction 
of {\em Herbrand sequents} and other objects which are essential in proof analysis.
Proof schemata preserve the ancestor relation by representing an inductive argument
through recursion indexed by a free parameter. Though, not completely satisfactory
for reductive methods, i.e.\ pushing cuts through proof links is non-trivial, the 
preservation of the ancestor relation benefits a more global approach to the problem. 
The schematic \CERES\ method~\cite{CERESS} takes advantage of this property and
results in a ``proof''  with the subformula property, though for a smaller class 
of inductive arguments than other methods. Problematically, the application of the 
method to proof schemata is not trivial. 

A method for the algorithmic extraction of the characteristic clause set from 
proof schemata exists, but like the proof schemata it is primitive recursively 
defined. As shown in~\cite{DBLP:conf/cade/CernaL16} refuting such a clause set is 
non-trivial even for a simple mathematically meaningful statement. This analysis was 
performed with the aid of automated theorem provers. Part of the issue is that 
the refutation must also be recursive. To simplify the construction of a 
recursive refutation, the substitutions are separated from the resolution rule and 
are given their own recursive definition. Essentially, we define the structure 
separately from the term instantiations.  Thus, the result of the method is a 
{\em substitution schema} and a {\em resolution refutation schema}, both having 
a free parameter, which when instantiated can be combined to construct an 
\LK-derivation. An \LK-proof is not constructed being that schematic proof projections 
are not extracted. Nonetheless, the resulting \LK-derivation contains enough 
information to allow extraction of a {\em Herbrand system}~\cite{CERESS}, that 
is a schematic version of a Herbrand sequent. Though, this separation simplifies 
the definitions and construction it requires a mapping to exists between the two 
schemata, thus, more complex refutations/term instantiation become difficult to 
represent using the method. This issue was highlighted in~\cite{MyThesis}. Thus,
finding a canonical structure~\cite{WolfThesis} for resolution refutations of 
proof schemata would remove the problem of mapping the substitutions to the 
refutation structure.

It has been shown that  characteristic clause sets extracted from  \LK-proofs are 
of a specific shape~\cite{BAAZ2006a,Baaz:2013:MC:2509679,WolfThesis}, thus 
resulting in interesting clausal subsumption properties. If we consider an 
alternative reductive cut-elimination method, which does not 
eliminate atomic cuts, it is possible to prove the existence of a subsumption relation
between the characteristic clause sets extracted at different steps during the application of 
reductive cut-elimination to an \LK-proof~\cite{BAAZ2006a,Baaz:2013:MC:2509679}. The 
characteristic clause set of the final proof, all atomic cuts as high as possible in 
the proof structure, is subsumed by all the extracted clauses sets from previous 
steps~\cite{BAAZ2006a,Baaz:2013:MC:2509679}. Essentially, there is a canonical 
clause set associated with the \CERES\ method~\cite{WolfThesis}. What is most important 
about this canonical clause set is that there is only one way to refute it. The structure
 of this refutation is recursive and fits perfectly into the current language for representing
schematic resolution refutations~\cite{ConduThesis}. Problematically, we do not know 
if the subsumption properties hold for schematic characteristic clause sets. In 
the case of propositional proof schema, a weakening rule was added to the 
resolution calculus allowing the construction of the so called 
{\em Top Atomic Cut Normal Form (\ACNFtop)}~\cite{ConduThesis,WolfThesis}. Note that 
resolution for propositional logic does not require substitutions and thus this 
provides a decidable and complete method for propositional proof schemata. 
While it seems intuitive that the result would carry over, the original clausal 
analysis of first-order sequent calculus proofs used the Gentzen 
method~\cite{Gentzen1935} to construct the subsumption relation between clause sets. As we have
already pointed out, reductive cut-elimination and proof links do not get along.

In this paper we develop a proof schema transformation which allows one to 
``unroll'' a proof link without instantiating the free parameter. The transformation 
preserves syntactic equivalence insomuch as for every instantiation of the 
resulting proof schema there is an instantiation of the original proof schema 
which results in an identical \LK-proof. This transformation allows us to 
perform the Gentzen method on proof schemata because every time we get to a 
proof link we apply the transformation. Though, this does not result in a method 
which leads to cut-free proofs, it allows us to extend the clausal analysis results 
of~\cite{BAAZ2006a,Baaz:2013:MC:2509679} to proof schemata and allows us to use the clause 
set of the {\em Top Atomic Cut Normal Form (\ACNFtop)}. This implies that we can 
use the schematic resolution refutation structure discussed in~\cite{ConduThesis} 
for every schematic characteristic clause set.   

%Unexpectedly, this transformation leads to speed up results between reductive 
%cut-elimination on proof schemata and instantiated proof schemata. For certain cuts 
%in a given proof 
%schema it is possible to apply polynomially many standard Gentzen rules at once. 
%This amounts to reduction of polynomially many cuts at once. Thus, it seems that 
%applying our method first to a proof schema, then instantiating it and finishing the 
%elimination using the standard Gentzen method can be much more efficient than 
%applying the Gentzen method to the instantiated proof schema. For very large proof 
%schema this can be extremely helpful for proof analysis. 

This work is part of a research program into computational proof analysis, which, 
so far, has lead to 
the analysis of F\"{u}rstenberg's proof concerning the 
infinitude of primes~\cite{Baaz:2008:CAF:1401273.1401552}. The proof 
includes an inductive argument and as mentioned earlier the standard \CERES\ method 
cannot be applied. A recursive formalization of the proof was constructed  allowing for a informal clausal proof analysis and the discovery of a relationship between Euclid's proof and  F\"{u}rstenberg's proof. Though, the theory of proof schemata is not yet sufficiently developed to handle such a complex 
proof formally, the results of this paper greatly advances the state of the art in the 
field. Though, investigations of proof schemata have so far been driven by the 
above research program, the applications are not limited to this particular type 
of proof analysis. They also serve as a compact way to store a formal proof 
whilst allowing analysis and transformations to take place. We plan to 
investigate alternative uses of the formalism in future work.  

The rest of this paper is as follows: In Section~\ref{sec:prelim} we discuss the necessary background knowledge needed for the results. In Section~\ref{sec:lazyIN} we introduce the concept of {\em lazy instantiation}. In Section~\ref{sec:clausAnalProofSch} we use lazy instantiation to perform clause analysis of proof schema and show that every schematic clause set can be subsumed by a sequence of clause sets in top form corresponding to the instantiations of the schematic clause set.  In Section~\ref{sec:Conclusion}, we conclude the paper and discuss future work and open problems. 

\section{Preliminaries}
\label{sec:prelim}
Due to the maturity of Gentzen's reductive cut-elimination and the sequent calculus we refrain from giving an introduction to the material. Unconventional 
uses of the method and/or calculus will be addressed when necessary to understanding. 
For more details, there exists numerous publications addressing his results 
(\cite{Baaz:2013:MC:2509679,Gentzen1935,prooftheory} to name a few) as well as the 
basics of the sequent calculus. One point we would like to address concerning reductive cut-elimination 
as considered in this paper is that we will use the generalized rewrite system introduced in~\cite{BAAZ2006a} allowing application of the rewrite rules to any cut in the proof, not just to the upper-most cut. This generalized rewrite rule system $\mathcal{R}$ will be referred to as {\em reductive} cut-elimination. When an application of reductive cut-elimination must be made explicit, we write $\psi \rightarrow_{\mathcal{R}} \chi $ where $\rightarrow_{\mathcal{R}}$ is a binary relation on \LK-derivations and $\psi$ and $\chi$ are  \LK-derivations. The meaning of $\psi \rightarrow_{\mathcal{R}} \chi $ is $\psi$ can be transformed into $\chi$ using the rewrite rule system  $\mathcal{R}$. By $\rightarrow_{\mathcal{R}}^{*}$, we mean the reflexive and transitive closure of $\rightarrow_{\mathcal{R}}$. 

\subsection{The  Atomic Cut Normal Form (\textit{top}) 
\label{sec:TANCF}
(\textbf{ACNF}$^{(\text{top})}$) }
The Atomic Cut Normal Form of an \LK-calculus proof was introduced in~\cite{CERES}\ as 
the result of the \CERES\ method. A proof transformed to 
\ACNF\ still has cuts, but only rank reduction has to be applied to the proof to get 
a truly cut-free proof, i.e.\ all cuts are atomic. Rank reduction operations are 
the cheapest reduction rules of $\mathbf{R}$ which may be applied to an \LK-proof, thus, 
a proof in \ACNF\ can be thought of as essentially cut-free. In this work, we consider 
a special \ACNF, the \ACNFtop.  An \ACNFtop\ is an \ACNF\ where all cuts are shifted 
to the top of the proof. An easy way to transform a given proof $\varphi$ into 
\ACNFtop\ is to apply generalized reductive cut-elimination  
without the following transformation rules

\begin{center}
\begin{tabular}{cc}
\begin{minipage}{0.45\textwidth}
\begin{prooftree}
\AxiomC{$A\vdash A$}
\AxiomC{$\psi_1$}
\noLine
\UnaryInfC{$A,\Gamma\vdash\Delta$}
\RightLabel{cut}
\BinaryInfC{$A,\Gamma\vdash\Delta$}
\end{prooftree}
\end{minipage} &
\begin{minipage}{0.45\textwidth}
\begin{prooftree}
\AxiomC{$\psi_2$}
\noLine
\UnaryInfC{$\Gamma\vdash\Delta ,A$}
\AxiomC{$A\vdash A$}
\RightLabel{cut}
\BinaryInfC{$\Gamma\vdash\Delta ,A$}
\end{prooftree}
\end{minipage} \\
 &  \\
$\Downarrow$ & $\Downarrow$ \\
\begin{minipage}{0.5\textwidth}
\begin{prooftree}
\AxiomC{$\psi_1$}
\noLine
\UnaryInfC{$A,\Gamma\vdash\Delta$}
\end{prooftree}
\end{minipage} &
\begin{minipage}{0.5\textwidth}
\begin{prooftree}
\AxiomC{$\psi_2$}
\noLine
\UnaryInfC{$\Gamma\vdash\Delta ,A$}
\end{prooftree}
\end{minipage} \\
 & 
\end{tabular}
\end{center}
being applied to $\varphi$. A full explanation can be found in \cite{BAAZ2006a,WolfThesis}. 

\subsection{Proof Schemata: Language, Calculus, and Interpretation}

To give a formal construction of proof schemata we extend the classical first-order term language. Following
the syntax of a to be published paper concerning proof schemata \cite{schema2017Journal}, we partition the set of function symbols $\mathbf{P}$ into two categories, {\em uninterpreted function symbols} $\mathbf{P_{u}}$ and {\em defined function symbols} $\mathbf{P_{d}}$. Together with a countable set of {\em schematic variable symbols} these symbols are used to construct the $\iota$ sort of the term language. We will denote  defined function symbol using $\hat{\cdot}$ in order to distinguish them from uninterpreted function symbols. Defined function symbols are added to the language to allow the definition of primitive recursive functions within the object language. Analogously, we allow {\em uninterpreted predicate symbols} and {\em defined predicate symbols}. 

In addition to the $\iota$ sort,  we add an $\omega$ sort where every term normalizes to numerals (hence, 
the only uninterpreted function symbols are $0$ and $s(\cdot)$). We will denote numerals by lowercase Greek 
letters, i.e. $\alpha$, $\beta$, $\gamma$, etc. There is also a countable set $\mathcal{N}$ of {\em 
parameter} symbols of type $\omega$. For this work, we will only need a single parameter which in most cases 
we denote by $n$. We use  $k,k'$ to represent $\omega$-terms containing the parameter. This parameter 
symbol, referred to as the {\em free parameter},  is used to index \LKS-derivations. The set of all 
$\omega$-terms will be denoted by $\fat{\Omega}$. Schematic variables are of type $\omega\to\iota$, i.e.\ 
when the schematic variable $x$ is instantiated by a numeral $\alpha$, $x(\alpha)$, it is to be interpreted 
as a first-order variable of type $\iota$. Hence, schematic variable symbols describe infinite sequences of 
distinct variables.

 Formula schemata (a generalization of formulas including defined functions and predicates) are defined as usual by induction starting with a countable set of 
uninterpreted predicate symbols and defined predicate symbols. To deal with defined functions and predicates, we add to our interpretation a set of convergent rewrite rules, $\mathcal{E}$, of the form $\hat{f}(\bar{t}) \rightarrow E$, where $\bar{t}$ contains no defined symbols, and either $\hat{f}$ is a function symbol of range $\iota$ and $E$ is a term or $\hat{f}$ is a predicate symbol and E is a formula. We extend the \LK-calculus with an inference rule applying this set of rewrite rules as an equational theory. 
This extended calculus is referred to as \LKE.

\begin{definition}[\LKE]
Let $\mathcal{E}$ be a given equational theory. \LKE\ is an extension of \LK\ with 
the $\mathcal{E}$ inference rule
\begin{small}\begin{prooftree}
\AxiomC{$S(t)$}
\RightLabel{$\mathcal{E}$}
\UnaryInfC{$S(t')$}
\end{prooftree}\end{small}
where the term $t$ in the sequent $S$ is replaced by a term $t'$ for 
$\mathcal{E}\models t=t'$.
\end{definition}

\begin{example}
We can add iterated  $\lor$ and $\land$ ( the defined predicates are abbreviated 
as $\bigvee$ and $\bigwedge$ ) to the term language using the following rewrite rules: 
\begin{align*}
\bigvee_{i=0}^0P(i)\equiv & \bigwedge_{i=0}^0P(i)\equiv 
P(0), &
\bigvee_{i=0}^{s(y)}P(i)\equiv & \bigvee_{i=0}^{y}P(i)\lor P(s(y)), &
\bigwedge_{i=0}^{s(y)}P(i)\equiv &\bigwedge_{i=0}^{y}P(i)\land P(s(y))
\end{align*}

\end{example}

Formula schemata are used to build {\em schematic sequents} $\Delta\vdash \Gamma$, where  $\Delta$ and $\Gamma$ are multi-sets of formula schemata. Schematic sequents are to be interpreted in the standard way.  Note that the size of $\Delta$ and $\Gamma$ cannot depend on the free parameter $n$. 

Schematic proofs are a finite ordered list of {\em proof schema components} that can 
interact with each other. This interaction is defined using so-called 
\emph{proof links}, a 0-ary inference rule we add to \LKE-calculus: Let 
$S(k,\bar{x})$ be a sequent where $\bar{x}$ is a
vector of schematic variables.  By 
$S(k,\bar{t})$ we denote $S(k,\bar{x})$ where $\bar{x}$ is replaced by $\bar{t}$ 
respectively, where $\bar{t}$ is a vector of terms of appropriate type. 
Furthermore, we assume a countably infinite set of \emph{proof symbols} 
$\mathcal{B}$ denoted by $\varphi ,\psi ,\varphi_i,\psi_j$. The expression 

\begin{prooftree}
\AxiomC{$(\varphi,k,\bar{t})$}
\dashedLine
\UnaryInfC{$S(k,\bar{t})$}
\end{prooftree}

is called a proof link with the intended meaning that there is a proof called $\varphi$ with the
end-sequent $S(k,\bar{x})$. For  $k\in \fat{\Omega}$, let $\mathcal{V}(k)$ be the set of parameters in $k$.  We refer to a proof link as an \emph{$E$-proof link} if $\mathcal{V}(k) \subseteq E$. Note that in this work $E = \left\lbrace n \right\rbrace$ or $E = \emptyset$. 

\begin{definition}[\LKS]
\LKS\ is an extension of \LKE, where proof links may appear at the leaves of a proof. 
\end{definition}

The class of all \LKS-proofs will be denoted by $\lyus{2}$ (read \textit{little 
yus}). For two \LKS-proofs (\LKE-proof) $\varphi$ and $\psi$, we say that $\varphi \equiv_{syn} \psi$ if they are syntactically the same.  We will denote a \textit{sub}\LKS-proof (\LKE-proof) $\nu$ of an 
\LKS-proof (\LKE-proof) proof $\varphi$ by $\varphi.\nu$. If the 
\textit{sub}\LKS-proof $\varphi.\nu$ is a $\left\lbrace n \right\rbrace$-proof link $(\psi,k,\overline{x})$ 
then we define $\mathcal{P}(\varphi.\nu) = \psi$ and 
$\mathcal{I}(\varphi.\nu) = k$. If $\varphi.\nu$ is not a proof link $\mathcal{P}(\varphi.\nu)$ and $\mathcal{I}(\varphi.\nu)$ do nothing. The set of 
proof links in an \LKS-proof $\varphi$ is $\mathit{PL}(\varphi)$. We also define 
substitution of subproofs, that is let $\varphi,\psi$ be \LKS-proof and 
$\varphi.\nu$ a \textit{sub}\LKS-proof such that $\ES{\psi} = 
\ES{\varphi.\nu}$, 
then we can construct a new \LKS-proof $ \varphi\left[\psi \right]_{\nu}$ where 
$\varphi.\nu$ is replaced by $\psi$. We also allow term substitutions for the parameters of an \LKS-proof. 

\begin{definition}
Let $\nu \in \lyus{2}$ be an \LKS-proof. A {\em parameter substitution} is a term substitution $\sigma: \mathcal{N} \rightarrow  \fat{\Omega}$, such that $\nu\sigma$ is a new \LKS-proof with the substitutions applied to every $\omega$-term containing a parameter.  
\end{definition}

\begin{definition}[Proof Schema Component]
Let $\psi\in \mathcal{B}$ and $n\in \mathcal{N}$. A \emph{proof schema 
component} $\mathbf{C}$  is a triple $(\psi,\pi ,\nu(k))$  where $\pi$ is an 
\LKS-proof only containing $\emptyset$-proof links and $\nu(k)$ 
is an \LKS-proof containing $\left\lbrace n\right\rbrace $-proof links. The end-sequents of the proofs are  $S(0,\bar{x})$ and $S(k,\bar{x})$,
respectively. Given a proof schema 
component $\mathbf{C}=(\psi,\pi ,\nu(k))$ we define $\mathbf{C}.1 = \psi$, 
$\mathbf{C}.2 = \pi$, 
and  $\mathbf{C}.3 = \nu(k)$. 
\end{definition}

If $\nu(k)$ of a proof schema component $(\psi,\pi ,\nu(k))$  contains a proof link to $\psi$ it will be referred to as {\em cyclic}, otherwise it is {\em acyclic}. 

%\begin{definition}[Proof Schema Conditional]
%Let $\mathbf{C}_1,\cdots, \mathbf{C}_m$ be proof schema components such that 
%$\mathbf{C}_i.1$ is distinct for $1\leq i\leq m$ and for any $\mathbf{C}_i$ the end-sequents of $\mathbf{C}_i.2$ and $\mathbf{C}_i.3$ are $S(0,\bar{x})$ and $S(k,\bar{x})$ 
%respectively. Let $S_{1},\cdots S_{m} \subseteq \mathbb{N}$ such that $\bigcup_{i=0}^{m} S_{m} \equiv \mathbb{N}$ and for any $S_{i}$ and $S_{j}$, where $1\leq i<j\leq m$, $S_{i} \cap S_{j} \equiv \emptyset$. Then we define a proof schema conditional as the sequence $\boxed{\mathbf{C}_{1}:S_{1}\ \fvee \ \mathbf{C}_{2}:S_{2} \ \fvee \ \cdots \ \fvee \ \mathbf{C}_{m}:S_{m}}$. 
%\end{definition}

\begin{definition}[Proof Schema \cite{CERESS}]
Let $\mathbf{C}_1,\cdots, \mathbf{C}_m$ be proof schema components such that 
$\mathbf{C}_i.1$ is distinct for $1\leq i\leq m$ and $n\in \mathcal{N}$. Let the  end sequents of $\mathbf{C}_1$ be $S(0,\bar{x})$ and $S(k,\bar{x})$. We define  $\Psi 
= \left\langle \mathbf{C}_1, \cdots, \mathbf{C}_m \right\rangle $ as a 
{\em proof schema} if 
$\mathbf{C}_i.3$ only contains $\left\lbrace n \right\rbrace $-proof links to $\mathbf{C}_i.1$ or 
$\mathbf{C}_j.1$ for $1\leq i < j \leq m$. The $\left\lbrace n \right\rbrace $-proof links are of the 
following form: 

\begin{minipage}{.45\textwidth}
\begin{prooftree}
\AxiomC{$(\mathbf{C}_i.1, k',\bar{a})$}
\dashedLine
\UnaryInfC{$S'(k',\bar{a})$}
\end{prooftree}
\end{minipage}
\begin{minipage}{.45\textwidth}
\begin{prooftree}
\AxiomC{$(\mathbf{C}_j.1, t,\bar{b})$}
\dashedLine
\UnaryInfC{$S''(t,\bar{b})$}
\end{prooftree}
\end{minipage}

where $t\in \fat{\Omega}$ such that $\mathcal{V}(t)\in \left\lbrace n \right\rbrace$, $k'$ is a sub-term of $k$, and $\bar{a}$ and 
$\bar{b}$ are vectors of terms from the appropriate sort. $S'(k',\bar{a})$ 
and $S''(t,\bar{b})$ are the end sequents of components $\mathbf{C}_i$ and 
$\mathbf{C}_j$ respectively.  We call $S(k,\bar{x})$ the end sequent 
of $\Psi$ and assume an identification between the formula occurrences in the end 
sequents of the proof schema components so that we can speak of occurrences in the 
end sequent of $\Psi$. The class of all proof schemata will be denoted by $\byus{2}$ (read \textit{big 
yus}). 
\end{definition}

For any proof schema $\Phi\in\byus{2}$, such that $\Phi= \left\langle 
\mathbf{C}_1, \cdots, \mathbf{C}_m \right\rangle $ we define $|\Phi|= m$ and 
$\Phi.i = \mathbf{C}_i$ for $1\leq i \leq m$. Note that instead of using {\em proof schema pair}~\cite{CERESS,schema2017Journal} to define proof schemata we use proof schema components. The only difference is that 
proof schema components make the name explicit. All results concerning proof schemata built from proof schema pairs carry over for our above definition. 

\begin{definition}
Let $\Phi\in\byus{2}$ such that $\Phi$ contains $(\chi ,\pi,\nu(k))$. Let $\nu.\gamma \in PL(\nu(k))$. We call  $\nu.\gamma$ an \emph{incoming link} if $\mathcal{P}(\nu(k).\gamma ) = \chi $, otherwise, we call it an \emph{outgoing link}.
\end{definition}

\begin{example}
\label{example.1}
Let us consider the proof schema $\Phi = \left\langle (\varphi,\pi,\nu (k)) \right\rangle$. The proof schema uses one defined function symbol $\hat{S}(\cdot)$ which is used to convert terms of the $\omega$ sort to the $\iota$ sort, i.e. $\mathcal{E} = \left\lbrace \hat{S}(k+1) = s(\hat{S}(k)) \ ; \ \hat{S}(0) = 0 \ ; \ k+s(l)=s(k+l)\right\rbrace $. The proofs $\pi$ and $\nu (k)$ are as follows: 
\begin{flushleft}
\begin{minipage}{.2\textwidth}
$\pi =$
\end{minipage}
\begin{minipage}{.7\textwidth}
\begin{small}
\begin{prooftree}
\AxiomC{$P(\alpha+0) \vdash P(\alpha +0)$}
\RightLabel{$w\colon l$}
\UnaryInfC{$P(\alpha+0) ,\forall x.P(x)\to P(s(x)) \vdash P(\alpha +0)$}
\RightLabel{$\mathcal{E}$}
\UnaryInfC{$P(\alpha +0),\forall x.P(x)\to P(s(x)) \vdash P(\alpha +\hat{S}(0))$}
\end{prooftree}
\end{small}
\end{minipage}
\end{flushleft}
\begin{flushleft}
\begin{minipage}{.07\textwidth}
$\nu (k) =$
\end{minipage}
\begin{minipage}{.9\textwidth}
\begin{small}
\begin{prooftree}
\AxiomC{$(\varphi ,n,\alpha)$}
\dashedLine
\UnaryInfC{$P(\alpha+0),\forall x.P(x)\to P(s(x)) \vdash P(\alpha + \hat{S}(n) )$}
\AxiomC{$P(s(\alpha + \hat{S}(n))) \vdash P(s(\alpha + \hat{S}(n)))$}
\RightLabel{$\to\colon l$}
\BinaryInfC{$P(\alpha+0),\forall x.P(x)\to P(s(x)) ,P(\alpha + \hat{S}(n) )\to P(s(\alpha + \hat{S}(n) )) 
\vdash P(s(\alpha + \hat{S}(n) ))$}
\RightLabel{$\forall\colon l$}
\UnaryInfC{$P(\alpha+0),\forall x.P(x)\to P(s(x)) ,\forall x.P(x)\to P(s(x)) \vdash P(s(\alpha + \hat{S}(n) ))$}
\RightLabel{$\mathcal{E}$}
\UnaryInfC{$P(\alpha+0),\forall x.P(x)\to P(s(x)),\forall x.P(x)\to P(s(x)) \vdash  P(\alpha + s(\hat{S}(n)))$}
\RightLabel{$\mathcal{E}$}
\UnaryInfC{$P(\alpha+0),\forall x.P(x)\to P(s(x)),\forall x.P(x)\to P(s(x)) \vdash P(\alpha + \hat{S}(n+1))$}
\RightLabel{$c\colon l$}
\UnaryInfC{$P(\alpha+0),\forall x.P(x)\to P(s(x)) \vdash P(\alpha + \hat{S}(n+1))$}
\end{prooftree}
\end{small}
\end{minipage}
\end{flushleft}
Note that $\pi$ contains no proof links, while $\nu (k)$ contains $\{ n\}$-proof links.
\end{example}

For our clausal analysis, we need to consider a specific type of proof schema which we will refer to as an {\em accumulating proof schema}. For a given proof schema we can easily construct the corresponding accumulating proof schema by adding a new proof component to the beginning of the proof schema. The additional component allows incremental unfolding of the proof schema (see Section \ref{sec:lazyIN}).

\begin{definition}
\label{def:AccSch}
Let $\Phi = \left\langle \mathbf{C}_1, \cdots, \mathbf{C}_m \right\rangle \in \byus{2}$. We define the accumulating proof schema of $\Phi$, as $\Phi^{a} = \left\langle A, \mathbf{C}_1, \cdots, \mathbf{C}_m \right\rangle$, where $A = \left(\chi , \pi , \nu(k) \right)$, $\chi$ is a distinct proof symbol from those of $\Phi$, $\pi = \mathbf{C}_{1}.2$, and $\nu(k)$ is defined as follows: 
\begin{center}
\begin{small}\begin{prooftree}
\AxiomC{$\begin{array}{c} (\mathbf{C}_1.1,k,\bar{t})\end{array}$}
\dashedLine
\UnaryInfC{$\begin{array}{c} S(k,\bar{t})\end{array}$}
\end{prooftree}\end{small}
\end{center}
where $S(k,\bar{t})$ is the end sequent of $\mathbf{C}_{1}.3$. 
\end{definition} 

\begin{definition}[Evaluation of proof schema~\cite{CERESS}]
\label{def:eval}
We define the rewrite rules for proof links 
\begin{flushleft}
\begin{minipage}{.45\textwidth}
\begin{small}\begin{prooftree}
\AxiomC{$(\varphi ,0,\bar{t})$}
\dashedLine
\RightLabel{$\Rightarrow\pi$}
\UnaryInfC{$S(0,\bar{t})$}
\end{prooftree}\end{small}
\end{minipage}
\begin{minipage}{.45\textwidth}
\begin{small}\begin{prooftree}
\AxiomC{$(\varphi ,k,\bar{t})$}
\dashedLine
\RightLabel{$\Rightarrow\nu (k)$}
\UnaryInfC{$S(k,\bar{t})$}
\end{prooftree}\end{small}
\end{minipage}
\end{flushleft}
for all proof schema components $\mathbf{C}=(\varphi ,\pi ,\nu (k))$. We define $\mathbf{C}\downarrow_{\alpha}$ as a normal form of 
\begin{small}\begin{prooftree}
\AxiomC{$(\varphi ,\alpha,\bar{t})$}
\dashedLine
\UnaryInfC{$S(\alpha,\bar{t})$}
\end{prooftree}\end{small}
where $\alpha\in \mathbb{N}$, under the rewrite system just given extended with rewrite rules for defined function 
and predicate symbols. Further, we define $\Phi\downarrow_{\alpha} = 
\mathbf{C}_1\downarrow_{\alpha}$ for a proof schema $\Phi =\langle\mathbf{C}_1,\ldots 
,\mathbf{C}_m\rangle$. 
\end{definition}

\begin{example}
Let $\Phi$ be the proof schema of example \ref{example.1} and $\Phi^a$ be the corresponding accumulating proof schema. For $1\in\mathbb{N}$ we can write down $\Phi^a\downarrow_1$:
\begin{center}
\begin{small}
\begin{prooftree}
\AxiomC{$P(s(\alpha + S(0))) \vdash P(s(\alpha + S(0)))$}
\AxiomC{$P(\alpha + 0)\vdash P(\alpha + 0)$}
\RightLabel{$w\colon l$}
\UnaryInfC{$P(\alpha + 0) ,\forall x.P(x)\to P(s(x)) \vdash P(\alpha + 0)$}
\RightLabel{$\mathcal{E}$}
\UnaryInfC{$P(\alpha+0),\forall x.P(x)\to P(s(x)) \vdash P(\alpha + S(0))$}
\RightLabel{$\to\colon l$}
\BinaryInfC{$P(\alpha+0),\forall x.P(x)\to P(s(x)) ,P(\alpha + S(0) )\to P(s(\alpha + S(0) )) 
\vdash P(s(\alpha + S(0)))$}
\RightLabel{$\forall\colon l$}
\UnaryInfC{$P(\alpha+0),\forall x.P(x)\to P(s(x)) ,\forall x.P(x)\to P(s(x)) \vdash P(s(\alpha + S(0) ))$}
\RightLabel{$\mathcal{E}$}
\UnaryInfC{$P(\alpha+0),\forall x.P(x)\to P(s(x)),\forall x.P(x)\to P(s(x)) \vdash  P(\alpha + s(S(0)))$}
\RightLabel{$\mathcal{E}$}
\UnaryInfC{$P(\alpha+0),\forall x.P(x)\to P(s(x)),\forall x.P(x)\to P(s(x)) \vdash P(\alpha + S(0+1))$}
\RightLabel{$c\colon l$}
\UnaryInfC{$P(\alpha +0),\forall x.P(x)\to P(s(x)) \vdash P(\alpha + S(0+1))$}
\RightLabel{$\mathcal{E}$}
\UnaryInfC{$P(\alpha +0),\forall x.P(x)\to P(s(x)) \vdash P(\alpha + S(1))$}
\end{prooftree}
\end{small}
\end{center}
\end{example}

\begin{proposition}[Soundness of proof schemata \cite{CERESS}]
\label{prop:sound}
Let $\Phi=\langle\mathbf{C}_1,\ldots ,\mathbf{C}_m\rangle$ be a proof schema, 
and let $\alpha\in \mathbb{N}$. Then there exists a \LK-proof of 
$\mathbf{C}_1\downarrow_{\alpha}$.
\end{proposition}

If $S(k,\bar{t})$ is the end sequent of $\Phi$ then Proposition~\ref{prop:sound} essentially states that $\mathbf{C}_1\downarrow_{\alpha}$ is an \LK-proof of $S(\alpha,\bar{t})\downarrow$ where by $\downarrow$ we refer to normalization of the defined symbols in $S(\alpha,\bar{t})$. 

\subsection{The Characteristic Clause Set}
The characteristic clause set of an \LK-proof $\varphi$ is extracted by inductively following the formula occurrences of {\em cut  ancestors} up the proof tree to the leaves. The cut ancestors are sub-formulas of any cut in the given proof. However, in the case of proof schemata, the concept of ancestors and formula occurrence is more complex. A formula occurrence might be an ancestor of a cut formula in one recursive call and in another it might not. Additional machinery is necessary to extract the characteristic clause set from proof schemata. A set $\Omega$ of formula occurrences from the end-sequent of an \textbf{LKS}-proof $\psi$ is called {\em a configuration for $\psi$}. A configuration $\Omega$ for $\psi$ is called relevant w.r.t. a proof schema $\Psi$ if $\psi$ is a proof in $\Psi$ and there is a $\gamma \in \mathbb{N}$ such that $\psi$ induces a subproof $\pi\downarrow_\gamma$ of $\Psi \downarrow_\gamma$
such that the occurrences in $\Omega$ correspond to cut-ancestors below $\pi\downarrow_\gamma$~\cite{thesis2012Tsvetan}. By $\pi \downarrow_\gamma$, we mean substitute the free parameter of $\pi$ with $\gamma\in \mathbb{N}$ and unroll the proof schema to an \textbf{LKE}-proof. We note that the set of relevant cut-configurations can be computed given a proof schema $\Psi$. To represent a proof symbol $\varphi$ and configuration $\Omega$ pairing in a clause set we assign a {\em clause set symbol} $cl^{\varphi,\Omega}(k,\bar{x})$ to them, where $k\in \fat{\Omega}$. 

\begin{definition}[Characteristic clause term \cite{CERESS}]
Let $\pi$ be an \LKS-proof and $\Omega$ a configuration. In the following, by 
$\Gamma_\Omega,\Delta_\Omega$ and $\Gamma_C,\Delta_C$ we will denote multisets of 
formulas of $\Omega$- and cut-ancestors respectively. Let $r$ be an inference 
in $\pi$. We define the clause-set term $\Theta_r^{\pi ,\Omega}$ inductively: 
\begin{itemize}
 \item if $r$ is an axiom of the form $\Gamma_\Omega ,\Gamma_C,\Gamma\vdash 
 \Delta_\Omega ,\Delta_C,\Delta$, then $\Theta_r^{\pi ,\Omega}=\{\Gamma_\Omega 
 ,\Gamma_C\vdash\Delta_\Omega ,\Delta_C\}$
 \item if $r$ is a proof of the form 
 \begin{prooftree}
 \AxiomC{$\psi (k,\bar{a})$}
 \dashedLine
 \UnaryInfC{$\Gamma_\Omega ,\Gamma_C,\Gamma\vdash 
 \Delta_\Omega ,\Delta_C,\Delta$}
 \end{prooftree}
 then define $\Omega'$ as the set of formula occurrences from $\Gamma_\Omega 
 ,\Gamma_C\vdash\Delta_\Omega ,\Delta_C$ and $\Theta_r^{\pi ,\Omega}=\cL^{\psi ,\Omega '} 
 (k,\bar{a})$
 \item if $r$ is a unary rule with immediate predecessor $r'$, then 
 $\Theta_r^{\pi ,\Omega}=\Theta_{r'}^{\pi ,\Omega}$
 \item if $r$ is a binary rule with immediate predecessor $r_1,r_2$, then 
 \begin{itemize}
  \item if the auxiliary formulas of $r$ are $\Omega$- or cut-ancestors, then 
  $\Theta_r^{\pi ,\Omega}=\Theta_{r_1}^{\pi ,\Omega}\oplus\Theta_{r_2}^{\pi ,\Omega}$ 
  \item otherwise $\Theta_r^{\pi ,\Omega}=\Theta_{r_1}^{\pi ,\Omega}\otimes 
  \Theta_{r_2}^{\pi ,\Omega}$
 \end{itemize}
\end{itemize}
Finally, define $\Theta^{\pi ,\Omega}=\Theta_{r_0}^{\pi ,\Omega}$, where $r_0$ is the 
last inference of $\pi$, and $\Theta^\pi =\Theta^{\pi ,\emptyset}$. $\Theta^\pi$ is 
called the characteristic term of $\pi$. 
\end{definition}
Note that considering an empty configuration and ignoring the rule for proof links results in the definition of clause set term for \LK-proofs~\cite{CERES}. Clause terms can be evaluate to sets of clauses using $|\Theta^{\pi}|$, that is $|\Theta_1 \oplus \Theta_2| = |\Theta_1| \cup |\Theta_2|$, $|\Theta_1 \otimes \Theta_2| = \{C \circ D \mid C \in |\Theta_1|, D \in |\Theta_2|\}$, where $C\circ D$ is the sequent whose antecedent is the union of the antecedents of $C$ and $D$ and analogously whose succedent is the union of the succedents of $C$ and $D$.

The characteristic clause term is extracted for each proof symbol in a given proof schema $\Psi$, and together they make the {\em characteristic clause set schema} for $\Psi$, $CL(\Psi)$.
\begin{definition}[Characteristic Term Schema\cite{CERESS}]
\label{def:CTS}
Let $\Psi = \left\langle \mathbf{C}_{1},\cdots, \mathbf{C}_{m} \right\rangle $ be a proof schema. We define 
the rewrite rules for clause-set symbols for all proof schema components $\mathbf{C}_{i}$  and 
configurations $\Omega$ as $cl^{\mathbf{C}_{i},\Omega}(0,\overline{u}) \rightarrow \Theta^{\pi_{i},\Omega}$ 
and $cl^{\mathbf{C}_{i},\Omega}(k,\overline{u}) \rightarrow \Theta^{\nu_{i},\Omega}$ where $1\leq i\leq 
m$. Next, let $\gamma\in \mathbb{N}$ and $cl^{\mathbf{C}_{i},\Omega}\downarrow_{\gamma}$ be the normal form 
of $cl^{\mathbf{C}_{i},\Omega}(\gamma,\overline{u})$ under the rewrite system just given extended by rewrite 
rules for defined function and predicate symbols. Then define $\Theta^{\mathbf{C}_{i},\Omega} = 
cl^{\mathbf{C}_{i},\Omega}$ and $\Theta^{\Psi,\Omega} = cl^{\mathbf{C}_{1},\Omega}$ and finally the 
characteristic term schema $\Theta^{\Psi} =  \Theta^{\Psi,\emptyset}$.
\end{definition}

Note that, concerning the characteristic clause schema of a proof schema $\Phi$ the following essential proposition holds. 

\begin{proposition}[\cite{CERESS,schema2017Journal}]
\label{prop:clauseequal}
Let $\gamma\in \mathbb{N}$. Then $\Theta^{\Phi}\downarrow_{\gamma} \equiv_{syn} \Theta(\Phi \downarrow_{\gamma})$
\end{proposition}
\begin{example} The following characteristic term schema is taken from the proof analysis of~\cite{DBLP:conf/cade/CernaL16} where $\Theta^{\Phi} = cl^{\psi,\emptyset}(n+1)$ and $\Omega(n)\equiv \exists x \forall y \left( \left( \left( x\leq y  \right) \rightarrow n+1= f(y) \right)    \vee f(y) < n+1 \right):$
\begin{align*}
{\scriptstyle cl^{\psi,\emptyset}(0)\equiv}& \begin{array}{l}  {\scriptstyle cl^{\varphi,\Omega(0)}(0) \oplus   \left(\left\lbrace\vdash f(\alpha)<0 \right\rbrace  \otimes \left\lbrace \vdash 0=f(\alpha)\right\rbrace  \otimes \left\lbrace  0\leq \beta \vdash \right\rbrace \right)}  \end{array}\\
{\scriptstyle cl^{\psi,\emptyset}(n+1)\equiv }& \begin{array}{l} {\scriptstyle cl^{\varphi,\Omega(n+1)}(n+1)\oplus   \left(\left\lbrace \vdash f(\alpha)<n+1 \right\rbrace   \otimes \left\lbrace  \vdash n+1=f(\alpha) \right\rbrace \otimes   \left\lbrace 0\leq \beta \vdash \right\rbrace  \right) } \end{array}\\
{\scriptstyle cl^{\varphi,\Omega(0)}(0) \equiv }& \begin{array}{l}{\scriptstyle \left\lbrace  f(\alpha)<0\vdash \right\rbrace \oplus\left\lbrace  f(g(\alpha))<0\vdash \right\rbrace  \oplus \left\lbrace  \vdash \alpha\leq\alpha \right\rbrace   \oplus \left\lbrace \vdash \alpha\leq g(\alpha)\right\rbrace    \oplus \left\lbrace  0=f(\alpha), 0=f(g(\alpha)) \vdash \right\rbrace } \end{array} \\
 {\scriptstyle cl^{\varphi,\Omega(n+1)}(n+1) \equiv }& \begin{array}{l}   {\scriptstyle cl^{\varphi,\Omega(n)}(n) \oplus  \left\lbrace n+1=f(\alpha),n+1=f(g(\alpha)) \vdash \right\rbrace  \oplus \left\lbrace \vdash \alpha\leq\alpha\right\rbrace  \oplus \left\lbrace  \alpha\leq g(\alpha)\right\rbrace  \oplus    \left\lbrace  n+1=f(\beta)\vdash n+1=f(\beta)\right\rbrace  \oplus } \\ {\scriptstyle  \left\lbrace \alpha\leq\beta \vdash \alpha\leq\beta \right\rbrace \oplus  \left\lbrace f(\beta)<n+1 \vdash f(\beta)<n+1\right\rbrace  \oplus  \left\lbrace f(\alpha)<n+1,\alpha\leq\beta \vdash n=f(\beta),f(\beta)<n  \right\rbrace   } \end{array}
\end{align*}
The following clause set is $\Theta^{\Phi}\downarrow_{1}$. Equivalence can be checked by referring to the formal proof of~\cite{DBLP:conf/cade/CernaL16}.
$$\begin{array}{lll}
\vdash \alpha  \leq  \alpha & \vdash \alpha \leq g(\alpha) & 0 = f(\alpha) , 0 = f(g(\alpha)) \vdash \\ 
1 = f(\alpha) , 1 = f(g(\alpha)) \vdash & \beta \leq \alpha , f(\beta)< 1 \vdash  f(\alpha)< 0 , 0 = f(\alpha) &f(\alpha)< 0\vdash    \\
f(g(\alpha))< 0\vdash  & 0\leq \alpha \vdash f(\alpha)< 1 , f(\alpha) = 1
\end{array}$$
\end{example}
\subsection{Top Clause Set}
Let us consider the characteristic clause set of an \ACNFtop\ $\varphi$, that is $\mathit{CL}(\varphi)$. The smallest unsatisfiable subset of $\mathit{CL}(\varphi)$ has a special structure pointed out be S. Wolfsteiner in \cite{WolfThesis} (Definition 6.1.6)  which he called a $\mathbf{TANCF}$. Such a clause set term is generated from the product of atomic clause sets $\left\lbrace A \vdash , \vdash A  \right\rbrace $ where $A$ is an atom. Later this structure was used for constructing clause sets from a set of propositional symbols by A. Condoluci~\cite{ConduThesis}, the so called {\em top clause set}. 

\begin{definition}[\cite{ConduThesis}]
Let $A$ be an  atom whose predicate symbol is $P\in \mathbf{P_{u}}$, and $\mathcal{A},\mathcal{A}'$ be
multisets  of atoms. We define the operator $\mathit{CL}^{t}(\cdot)$ which
maps a set of atoms to a clause set: 

\begin{minipage}{.4\textwidth}
$$\mathit{CL}^{t}(\lbrace A \rbrace) = \left\lbrace A\vdash , \vdash A  \right\rbrace $$
\end{minipage}
\begin{minipage}{.6\textwidth}
$$\mathit{CL}^{t}(\mathcal{A} \cup \mathcal{A}' ) = \mathit{CL}^{t}(\mathcal{A}) \times \mathit{CL}^{t}(\mathcal{A}')$$
\end{minipage}\\\\
 We refer to $\mathit{CL}^{t}(\mathcal{A})$ as the top clause set with respect to the set of atoms $\mathcal{A}$. 
\end{definition}

As was shown in both~\cite{ConduThesis,WolfThesis}, such a clause set is always unsatisfiable, even in the first-order case. We use this concept later to further extend the results presented in Section~\ref{sec:clausAnalProofSch}.

\subsection{Clausal Subsumption: Theory and Results}

Let $A(\bar{x})$ be a formula with the free variables $\bar{x}=(x_1,\ldots 
,x_n)$ for $n\in\mathbb{N}$. A substitution $\theta =[\bar{x}\backslash\bar{t}]$ is an 
instantiation of the free variables $x_1,\ldots ,x_n$ within a formula $A$ with the 
terms $\bar{t}=(t_1,\ldots ,t_n)$ such that $t_i$ do not contain $x_i$ for $1\le i
\le n$ as a subterm, i.e.\ $A(\bar{x})\theta =A(\bar{t})\theta$. 

We call a sequent containing only atomic formulas a {\em clause}. Moreover, we 
define the subset relation $\subseteq$ on clauses, i.e.\ a clause $C$ is a subset 
of a clause $D$ ($C\subseteq D$) if the atoms in the succedent of $C$ are a subset 
of the atoms in the succedent of $D$ and analogously the atoms in the antecedent of 
$C$ are a subset of the atoms in the antecedent of $D$. Now we are able to define 
the subsumption relation of clauses (see \cite{Baaz:2013:MC:2509679}). 

\begin{definition}[Subsumption]
Let $C$ and $D$ be clauses. Then $C$ subsumes $D$ ($C\SU D$) if there exists a 
substitution $\theta$ s.t. $C\theta\subseteq D$. 
Since we define $\subseteq$ with sets we do not have to consider permutation. 
We say that a set of clauses $\mathcal{C}$ subsumes a set of clauses $\mathcal{D}$ 
($\mathcal{C}\SU\mathcal{D}$) if for all $D\in\mathcal{D}$ there exists a $C\in 
\mathcal{C}$ s.t.\ $C\SU D$. 
\end{definition}

The concept of clausal subsumption was used in~\cite{BAAZ2006a} derive the theorems that the following two corollaries are based on. 

\begin{corollary}[Direct corollary of Theorem 6.1~\cite{BAAZ2006a}]
\label{cor:subTANCF}
Let $\varphi$ be an \LK-proof and $\psi$ be the \ACNFtop of $\varphi$ as described 
in Section~\ref{sec:TANCF}. Then $\Theta (\varphi )\SU\Theta (\psi )$.
\end{corollary}

\begin{corollary}[Direct corollary of Theorem 6.2~\cite{BAAZ2006a}]
Let $\varphi$ be an \LK-proof and $\psi$ be an \ACNFtop of $\varphi$ as described 
in Section~\ref{sec:TANCF}. Then there exists a resolution refutation $\gamma$ of $CL(\varphi )$ s.t.\ $\gamma\SU \textit{RES}(\psi )$ where $\textit{RES}(\psi )$ is the resolution proof corresponding to $\psi$.
\end{corollary}

\section{Lazy Instantiation of Proof Schema }
\label{sec:lazyIN}
In this section we define a transformation which allows for a kind of lazy 
evaluation~\cite{Henderson:1976:LE:800168.811543} of proof schemata, that is 
{\em lazy instantiation}. A similar concept was used by C. Sprenger and M. Dam \cite{mucalcyclic} to transform global induction into local induction for a variant of the $\mu$-calculus. Lazy instantiation  allows us to apply reductive 
cut-elimination to proof schemata by moving the cyclic structure (proof links) 
whenever it impedes the cut-elimination procedure, i.e.\ auxiliary sequents 
introduced by proof links do not have a reductive cut-elimination rule. Though one can never eliminate the cuts through lazy 
instantiation we can get a better understanding of the relationship between 
proof schemata in there instantiated and uninstantiated forms. Essentially, we 
extend the clausal analysis of~\cite{BAAZ2006a} to proof schemata. The lazy 
instantiation relation is defined as $\rightarrow_{L}\, :\byus{2}\rightarrow \byus{2}$. We 
first provide an example transformation before providing a formal definition of the transformation. 

\begin{example}
\label{ex:trans}
 Consider the following proof schema $\Phi = \left\langle \mathbf{C}_{1} \right\rangle$ where  $
\mathbf{C}_{1} = 
(\psi,\pi,\nu(k))$ (note that by $\_$ we represent the empty tuple of terms),

\begin{flushleft}
\begin{minipage}{.1\textwidth}
$$\nu(k)=$$
\end{minipage}

\begin{minipage}{.8\textwidth}
\begin{small}
\begin{prooftree}
\AxiomC{$\begin{array}{c} 
\varphi_{(n+1)} 
\end{array}$}
\AxiomC{$\begin{array}{c} 
(\psi,n,\_) 
\end{array}$}
\dashedLine
\UnaryInfC{$\begin{array}{c} 
\bigwedge_{i=0}^{n}  P(i) \rightarrow Q(i) 
\vdash 
\bigwedge_{i=0}^{n} \neg P(i) \vee Q(i) 
\end{array}$}
\RightLabel{$\wedge:r$}
\BinaryInfC{$\begin{array}{c} 
\bigwedge_{i=0}^{n+1}  P(i) \rightarrow Q(i) , P(n+1) \rightarrow Q(n+1)  
\vdash 
\left( \bigwedge_{i=0}^{n}  \neg P(i) \vee Q(i) \right) \wedge\left( \neg P(n+1) \vee Q(n+1) \right) 
\end{array}$}
\RightLabel{$\mathcal{E}$}
\UnaryInfC{$\begin{array}{c} 
\bigwedge_{i=0}^{n+1}  P(i) \rightarrow Q(i) , P(n+1) \rightarrow Q(n+1) 
\vdash 
\bigwedge_{i=0}^{n+1} \neg P(i) \vee Q(i) 
\end{array}$}
\RightLabel{$\land\colon l$}
\UnaryInfC{$\begin{array}{c} 
\left(  \bigwedge_{i=0}^{n}  P(i) \rightarrow Q(i)\right)\wedge  P(n+1) \rightarrow Q(n+1) 
\vdash 
\bigwedge_{i=0}^{n+1} \neg P(i) \vee Q(i)  
\end{array}$}
\RightLabel{$\mathcal{E}$}
\UnaryInfC{$\begin{array}{c} 
\bigwedge_{i=0}^{n+1}  P(i) \rightarrow Q(i) 
\vdash 
\bigwedge_{i=0}^{n+1} \neg P(i) \vee Q(i) 
\end{array}$}
\end{prooftree}
\end{small}
\end{minipage}
\end{flushleft}

where $\varphi_{(n+1)}$ is an \LK-proof of $P(n+1) \rightarrow Q(n+1) \vdash \neg P(n+1) \vee Q(n+1). $

\begin{flushleft}
\begin{minipage}{.1\textwidth}
$$\pi=$$
\end{minipage}
\begin{minipage}{.8\textwidth}
\begin{small}
\begin{prooftree}
\AxiomC{$\begin{array}{c}   P(0) \vdash  P(0) \end{array}$}
\RightLabel{$\neg:r$}
\UnaryInfC{$\begin{array}{c}  \vdash  \neg P(0), P(0) \end{array}$}
\AxiomC{$\begin{array}{c} Q(0) \vdash Q(0) \end{array}$}
\RightLabel{$\rightarrow:l$}
\BinaryInfC{$\begin{array}{c} P(0) \rightarrow Q(0) \vdash \neg P(0), Q(0)\end{array}$}
\RightLabel{$\vee:r$}
\UnaryInfC{$\begin{array}{c} P(0) \rightarrow Q(0) \vdash \neg P(0) \vee Q(0)\end{array}$}
\end{prooftree}
\end{small}
\end{minipage}
\end{flushleft}

We can lazily evaluate the accumulating proof schema $\Phi^a$ to $\Psi$, $\Phi^a\rightarrow_{L}\Psi$, where $\Psi = \left\langle \mathbf{C}_{1}',\mathbf{C}_{2}\right\rangle$,  $\mathbf{C}_{1}' = 
(\psi,\pi,\nu(k)')$, and  $\mathbf{C}_{2} = (\chi,\pi,\nu (k)'')$. Let $\nu(k).\mu = (\psi,n,\_)$. We define  $\nu(k)' = \nu(k)\sigma\left[\nu (k)'' \right]_{\mu}$, where $\sigma = \left\lbrace n\leftarrow n+1\right\rbrace$. The \LKS-proof $\nu (k)''$ is the same as $\nu(k)$ except  $\nu (k)''.\mu = (\chi,n,\_)$. The important parts of $\nu(k)'$ can be sketched as follows:

\begin{center}
\begin{small}
\begin{prooftree}
\AxiomC{$\begin{array}{c}\varphi_{(n+2)}\end{array}$}
\AxiomC{$\begin{array}{c}\varphi_{(n+1)}\end{array}$}
\AxiomC{$\begin{array}{c} (\chi,n,\_)\end{array}$}
\dashedLine
\UnaryInfC{$\begin{array}{c} 
\bigwedge_{i=0}^{n}  P(i) \rightarrow Q(i) 
\vdash 
\bigwedge_{i=0}^{n} \neg P(i) \vee Q(i) 
\end{array}$}
\RightLabel{$\wedge:r$}
\BinaryInfC{$\begin{array}{c} 
\vdots 
\end{array}$}
\RightLabel{$c\colon l$}
\UnaryInfC{$\begin{array}{c} 
\bigwedge_{i=0}^{n+1} P(i) \rightarrow Q(i) 
\vdash 
\bigwedge_{i=0}^{n+1} \neg P(i) \vee Q(i) 
\end{array}$}
\RightLabel{$\land\colon r$}
\BinaryInfC{$\begin{array}{c} 
\bigwedge_{i=0}^{n+1}  P(i) \rightarrow Q(i) ,  P(n+2) \rightarrow Q(n+2) 
\vdash 
\left( \bigwedge_{i=0}^{n+1}  \neg P(i) \vee Q(i) \right) \wedge\left( \neg P(n+2) \vee Q(n+2) \right) 
\end{array}$}
\RightLabel{$\mathcal{E}$}
\UnaryInfC{$\begin{array}{c} 
\vdots 
\end{array}$}
\RightLabel{$c\colon l$}
\UnaryInfC{$\begin{array}{c} 
\bigwedge_{i=0}^{n+2}  P(i) \rightarrow Q(i) 
\vdash 
\bigwedge_{i=0}^{n+2} \neg P(i) \vee Q(i) 
\end{array}$}
\end{prooftree}
\end{small}
\end{center}

Notice $\Phi^a\downarrow_{(s(s(\alpha)))} \equiv \Psi\downarrow_{s(\alpha)}$.
\end{example} 

For the formal definition we use $\nu (k)'$ and $\nu (k)''$ as in example \ref{ex:trans} even though $\nu (k)$ is not directly discussed. 

\begin{definition}
\label{def:LazyEval}
Let $\Phi,\Psi\in \byus{2}$, such that $\mathit{PL}(\Phi.1)$ does not have incoming links. We say that 
$\Phi$ {\em lazily instantiates} to $\Psi$, $\Phi \rightarrow_{L} \Psi$, if 
\begin{center}
\begin{itemize}
 \item $\Psi.i = \Phi.i$ for $1< i \leq |\Phi|$,
 \item $\Psi.1 = (\Phi.1.1, \Phi.1.2, \nu(k)')$ 
\end{itemize}
\end{center}
where $\nu (k)'$ is constructed from $\Phi.1.3$ as follows: First construct  $\nu (k)'' = \Phi.1.3\sigma$ where $\sigma=\lbrace n\leftarrow n+1\rbrace$ is a parameter substitution. Then we construct $\nu (k)'$ from $\nu (k)''$ by replacing every proof link $(\chi,k',\bar{t}) \in \textit{PL}(\nu (k)'')$ at position $\gamma$ with the appropriate \LKS-proof, that is if $\Phi.i = (\chi,\pi_i,\nu_i(k))$ then we substitute $\nu (k)''\left[ \nu_i(k)(k') \right]_{\gamma}$ and $\nu (k)''\left[ \pi_i \right]_{\gamma}$ if $k'=0$.
\end{definition}
\begin{remark}
Note that if $\Phi$ is accumulating  then $\mathit{PL}(\Phi .1)$ does not have incoming link. 
\end{remark}
In general it will be necessary to lazily instantiate a proof schema several times. In fact, we will consider all proof schemata that are producible by lazy instantiations.

\begin{definition}
Let $\Phi,\Psi\in \byus{2}$. We define the set of {\em lazily reachable proof schemata} from a proof schema $\Phi$, $\mathit{lr}(\Phi)$ by the following inductive construction:
\begin{itemize}
\item $\Phi\in \mathit{lr}(\Phi)$
\item if $\Psi \in \mathit{lr}(\Phi)$ and $\Psi\rightarrow_{L}\Gamma$, then $\Gamma\in \mathit{lr}(\Phi)$
\end{itemize}
\end{definition}

\begin{definition}
Let $\Phi,\Psi\in \byus{2}$. We define the set of {\em reductively reachable proof schemata} from a proof schemata $\Phi$, $\mathit{rr}(\Phi)$ by the following inductive construction ($\rightarrow_{R}$ is define at the beginning of Section~\ref{sec:prelim}) :
\begin{itemize}
\item $\Phi\in \mathit{rr}(\Phi)$
\item if $\Psi \in \mathit{rr}(\Phi)$ and $\Psi\rightarrow_{\mathcal{R}} \Gamma$, then $\Gamma\in \mathit{rr}(\Phi)$
\end{itemize}
\end{definition}
\begin{definition}
Let $\Phi,\Psi\in \byus{2}$. We define the {\em lr-chain length} $|\Psi|^{\Phi}_{lr}$(respectively the rr-chain length $|\Psi|^{\Phi}_{rr}$) inductively as follows:
\begin{itemize}
\item if $\Psi = \Phi$, then $|\Psi|^{\Phi}_{rr} = |\Psi|^{\Phi}_{lr} = 0 $, 
\item if $\Gamma \rightarrow_{L} \Psi$ or $\Gamma \rightarrow_{R} \Psi$ , then $|\Psi|^{\Phi}_{lr} = |\Gamma|^{\Phi}_{lr}+1$ $(|\Psi|^{\Phi}_{rr} = |\Gamma|^{\Phi}_{rr}+1)$ 
\item $|\Psi|^{\Phi}_{lr} = |\Psi|^{\Phi}_{rr} = \infty$ otherwise. 
\end{itemize}
\end{definition}
An important property of this transformation is that, when applied to accumulating proof schemata $\Phi$, for every proof schema in $\Psi\in \mathit{lr}(\Phi)$, there are only finitely many evaluations of $\Phi$ that are not equivalent to an evaluation of $\Psi$. In Example~\ref{ex:trans} there is no evaluation of $\Psi$ equivalent to $\Phi\downarrow_{1}$.

\begin{lemma}
\label{lem:pairing}
Let $\Phi\in \byus{2}$ be accumulating and $ \Psi \in  \mathit{lr}(\Phi)$. Then for all $ \alpha \in 
\mathbb{N}$  there exists $ \beta\in \mathbb{N}$ such that  $\Psi\downarrow_{\alpha} 
\equiv_{syn} \Phi\downarrow_{\beta}$.
\end{lemma}
\begin{proof}
By induction on $|\Psi|^{\Phi}_{lr}$.
\end{proof}

Lemma~\ref{lem:pairing} tells us that we can use lazy instantiation to change the structure of a given proof schema $\Phi$ into a proof schema $\Psi$ without completely losing the logical meaning, that is equivalence is only lost for finitely many \LK-proofs. More specifically, every evaluation of $\Psi$ is equivalent, syntactically, to an evaluation of $\Phi$.

\begin{definition}
Let $\Phi,\Psi\in \byus{2}$ such that $\Psi\in \mathit{lr}(\Phi)$ and $\Phi$ is accumulating. We define a {\em numeral pair} $(\alpha,\beta)$ as a pair of natural numbers such that $\Psi\downarrow_{\alpha}  \equiv_{syn} \Phi\downarrow_{\beta}$. The set of all numeral pairs for $\Psi$ and $\Phi$ is $\mathit{np}(\Psi,\Phi)$.
\end{definition}

The next lemma is a direct result of Lemma~\ref{lem:pairing}. 

\begin{lemma}
\label{lem:synlrrrproof}
Let $\Phi\in \byus{2}$ be an accumulating proof schema and $\Psi ,\Psi '\in\byus{2}$ be proof schemata. Furthermore, let $\Psi\in\mathit{lr}(\Phi )$ be lazily reachable from $\Phi$, $\Psi '\in\mathit{rr}(\Psi )$ be reductively reachable from $\Psi$, and $(\alpha,\beta)\in\mathit{np}(\Psi ,\Phi )$ be a numeral pair of $\Psi$ and $\Phi$. Then there exists an \LK-proof $\chi$ such that $\chi\in \mathit{rr}(\Phi\downarrow_{\beta})$ and $\Psi'\downarrow_{\alpha} \equiv_{syn} \chi$. 
\end{lemma}

\begin{proof}
Consider application of reductive cut-elimination rules which only reduce the cuts of $\Psi.1.3$. This means if a rule would require passing a cut through a proof link of $\Psi.1.3$ it cannot be applied. Being that $\Psi$ and $\Phi$ are equivalent when instantiated by $(\hat{n},\hat{m})$ and by definition $\Psi.1$ is acyclic, application of rules transforming $\Psi.1.3$ can be repeated in $\Phi$ without influencing unwanted  parts of the proof. The rest of the proof is an induction on  $|\Psi'|^{\Psi}_{rr}$.
\end{proof}

\section{Clausal Analysis of Proof Schemata}
\label{sec:clausAnalProofSch}

Before we discuss the relationship between the various concepts and results, as shown in Figure~\ref{fig:one}, we introduce a concept closely related to the \ACNFtop, that is the {\em atom set schema} and its clause set~\cite{ConduThesis}. Intuitively, an atom set schema of a clause set is just a list of the atoms found in a clause set schema. Though, not exactly the same as the clause set extracted from an \ACNFtop, the clause set of an atom set schema, when instantiated, is subsumed by the clause set of the corresponding \ACNFtop. 

\subsection{Atom Set Schema}
This concept was introduced by A. Condoluci in his master thesis concerning schematic propositional logic~\cite{ConduThesis}. We extend his definition to first-order logic by considering the predicate symbols modulo tuples of terms rather than just an arithmetic term. Note that, these terms form a schema which is easily extractable from a proof schema or characteristic clause set schema. We refer to this schema as a {\em predicate term schema}. In some cases the predicate term schema may be describable by a case distinction, i.e. the terms do not contain the free parameter.

\begin{definition}[Predicate Term Schema ]
Let $P\in \mathbf{P_{u}}$ be an $\alpha$-ary  predicate symbol. Then the { \em predicate term schema of $P$} is a  mapping $t_{P}: \mathbb{N}\rightarrow \iota^{\alpha}$ from numerals to $\alpha$-tuples of terms such that for $\beta\in \mathbb{N}$, $P(t_{P}(\beta))$ is a well-formed atom. we assume that, if the $\alpha$-tuple contains a defined function symbol it will be evaluated and normalized.
\end{definition}

\begin{definition}[Extension of Definition 78~\cite{ConduThesis}]
An {\em atom set schema} $\mathcal{A}$ is a finite set of 3-tuples
\begin{small}$$\mathcal{A} = \left\lbrace \Big\langle P_{1}, t_{P_{1}}^{1},a_{1}^{1}\Big\rangle , \cdots, \Big\langle P_{1},t_{P_{1}}^{\alpha_{1}},a_{1}^{\alpha_{1}}\Big\rangle ,\cdots, \Big\langle  P_{\beta}, t_{P_{\beta}}^{1},a_{\beta}^{1}\Big\rangle ,\cdots , \left\langle P_{\beta}, t_{P_{\beta}}^{\alpha_{\beta}},a_{\beta}^{\alpha_{\beta}}\right\rangle  \right\rbrace $$\end{small}
where $P_{i}\in \mathbf{P_{u}}$ for  $1\leq i\leq \beta$, $t_{P_{i}}^{j}$ is a predicate term schema of $P_{i}$ for $1\leq j\leq \alpha_{i}$, and $a_{i}^{j}$ is an arithmetic term. 
\end{definition}  

\begin{example}
\label{ex:atomschExt}
The atom set schema  $ \left\lbrace \left\langle P, \lambda r. (\alpha + \hat{S}(r)), n+1 \right\rangle  \right\rbrace $ was extracted from Example~\ref{example.1}.
\end{example}

\begin{definition}[Extension of Definition 79~\cite{ConduThesis}]
For every $\alpha \in \mathbb{N}$, we define the evaluation of an atom set schema, where $ \left\langle P, t_{P},a\right\rangle \downarrow{\alpha} = \left\lbrace  P(t_{P}(0)),P(t_{P}(1)),\cdots,  P(t_{P}(a\left[ n\setminus \alpha \right]))\right\rbrace$, as:
 
$$\mathcal{A}\downarrow_{\alpha} = \bigcup_{\left\langle P, t_{P},a \right\rangle \in \mathcal{A}} \left\langle P, t_{P},a\right\rangle \downarrow{\alpha}.$$

\end{definition}  

\begin{example}
Using the atom set schema of Example~\ref{ex:atomschExt}, we can derive the following evaluation:
$$ \left\lbrace \left\langle P, \lambda r. (\alpha + \hat{S}(r)), n+1 \right\rangle  \right\rbrace \downarrow_{5} = \left\lbrace  P(\alpha+0),P(\alpha+s(0)),\cdots, P(\alpha+s^5(0))\right\rbrace $$
\end{example}
 An atom set schema can always be extracted from a clause set term schema. 

\begin{lemma}[Completeness for atom set schemata \cite{ConduThesis}]
\label{lem:comass}
For every clause set term schema T, there exists an atom set schema for T. 
\end{lemma}
Though, this lemma was proven for schematic propositional logic, it can easily be extended by constructing the corresponding predicate term schema for each use of a given predicate symbol. 
\label{sec:ClauAnPs}
We introduced the concept of lazy instantiation to deal with precisely the following problem faced during application of reductive cut elimination to proof schemata. Note that this may result in additional predicate and term pairings which do not exists in the proof schema, but more importantly, the extraction method presented in~\cite{ConduThesis} guarantees that  no pairings, of which are presented in the given proof schema, are missing. This property allows us to further extend our subsumption results, as can be seen in Figure~\ref{fig:one} point (x). 

\begin{definition}[Top clause set schemata~\cite{ConduThesis}]
\label{def:topclass}
Let $\mathcal{A}$ be an atom set schema. We denote the corresponding top clause set schema by $\mathit{CL}^{t}(\mathcal{A})$. 
\end{definition}

\begin{definition}[Semantics of top clause set schemata~\cite{ConduThesis}]
For $\alpha\in \mathbb{N}$, $\mathit{CL}^{t}(\mathcal{A})\downarrow_{\alpha} = \mathit{CL}^{t}(\mathcal{A}\downarrow_{\alpha})$.
\end{definition}

In~\cite{ConduThesis}, it was shown that a propositional top clause set schemata can be refutated schematically (Lemma 10). Problematically, in the first-order case, a substitution schema must also be defined~\cite{CERESS}. If a substitution schema exists for a given first-order top clause set schema a schematic refutation can be defined. Though, the specification of such a substitution schema remains an open problem. In Figure~\ref{fig:one}, construction of the substitution schema  and construction of the refutation (see~\cite{ConduThesis}) are point (xii). 

\subsection{Clausal Analysis}
The main result of this paper is as follows:

\begin{theorem}
\label{thm:schclausalSub}
Let $\Psi\in \byus{2}{}$ be an accumulating proof schema, $\Phi\in lr(\Psi)$ be lazily reachable from $\Psi$, and $\Xi\in rr(\Phi)$ reductively reachable from $\Phi$. Then for all $\alpha \in \mathbb{N}$ there exists a numeral $\beta \in \mathbb{N}$ and an \ACNFtop $\chi$, such that $\Theta(\psi)\leq_{ss} \Theta(\chi)$ for all $\psi \in rr(\Psi\downarrow_{\beta})$, and $\Theta^{\Xi}\downarrow_{\alpha} \leq_{ss} \Theta(\chi)  \leq_{ss}  \mathit{CL}^{t}(\mathcal{A}^{\Theta^{\Xi}})\downarrow_{\alpha} $, where $\mathcal{A}^{\Theta^{\Xi}}$ is an atom set schema.
\end{theorem}
\begin{proof}
The statement $\Theta^{\Xi}\downarrow_{\alpha} \leq_{ss} \Theta(\chi) $ is a consequence of Lemma~\ref{lem:synlrrrproof} and Corollary~\ref{cor:subTANCF}. The statement $\Theta(\chi)  \leq_{ss}  \mathit{CL}^{t}(\mathcal{A}^{\Theta^{\Xi}})\downarrow_{\alpha} $ is a result of Lemma~\ref{lem:comass} and the extraction method of~\cite{ConduThesis}.
\end{proof}

In Figure~\ref{fig:one}, we present the entire process concerning how one goes from a proof schema $\Phi$ to 
a refutation. Note, as we mention before, the existence of a refutation for a first-order top clause set 
schema is dependent on the existence of a substitution schema~\cite{CERESS}, of which the construction is 
still an open problem, that is whether a substitution schema exists for every first-order top clause set. We 
leave this question to future work. The first step in Figure~\ref{fig:one}, (i), constructs an accumulating 
proof schema $\Phi^{A}$ from $\Phi$. This adds an additional component which ``accumulates'' the lazily 
instantiated proof links. After the construction of an accumulating proof schema $\Phi^{A}$, we can either 
instantiate it (iii) or apply lazy instantiation (ii) and construct a proof schema $\Phi'$. The point of 
lazy instantiation is to get around situation like the following: 
\begin{small}\begin{prooftree}
\AxiomC{$\begin{array}{c} (\varphi_{l},t,\bar{x})\end{array}$}
\dashedLine
\UnaryInfC{$\begin{array}{c}  C,\Delta \vdash  \Gamma\end{array}$}
\AxiomC{$\begin{array}{c} (\varphi_{j},t',\bar{x})\end{array}$}
\dashedLine
\UnaryInfC{$\begin{array}{c}  \Delta' \vdash  \Gamma', C\end{array}$}
\RightLabel{cut}
\BinaryInfC{$ \Delta,\Delta' \vdash  \Gamma,\Gamma'$}
\end{prooftree}\end{small}
in which Gentzen rules cannot be applied unless we evaluate the proof links. Note that by Lemma~\ref{lem:pairing}, we can instantiate $\Phi^{A}$ and $\Phi'$ such that they are equivalent. The heart of our result is that lazy instantiation allows us to apply reductive cut elimination to an arbitrarily large but finite portion of $\Phi^{A}$. This results in a proof schema $\Psi$. In turn, we can apply the same reductive cut elimination steps  $\Phi^{A}\downarrow_{\alpha}$ resulting in an  \LK-proof $\psi$ whose clause set subsumes the evaluation of the clause set schema of $\Psi$ at $\beta$. Furthermore, by extending the results of~\cite{BAAZ2006a}, we are able to show subsumption properties between the clause set schema of $\Psi$ and the clause set of $\psi$, as well as between both clause sets and the clause set of the \ACNFtop\ of $\psi$. 

 \begin{figure}
\begin{center}
\begin{small}\begin{tikzpicture}
    \node[] (A0) at (-5,1.5) {$\Phi $};
    \node[] (A) at (-3,1.5) {$\Phi^{a}$};
    \node[] (Bc) at (4.5,1.5) {$\Psi$};

    \node[] (B) at (0,1.5) {$\Phi'$};
    \node[] (An) at (-3,-.5)  {$\Phi^{a}\downarrow_{\alpha}$};
    \node[] (Bn) at (0,-.5)  {$\Phi'\downarrow_{\beta}$};
    \node[] (Bnc) at (4.5,0)  {$\Theta^{\Psi}$};
\node[] (Atomsch) at (6.5,0)  {$\mathcal{A}^{\Psi}$};
     \node[] (C) at (0,-2)  {$\psi$};
      \node[] (down4c) at (2,-2)  {$\Theta(\Psi\downarrow_{\beta})$};

    \node[] (Cr) at (4.5,-2)  {$\Theta^{\Psi}\downarrow_{\beta}$};

     \node[] (D) at (0,-4)  {$\psi^{top}$};
     \node[] (E) at (2,-4)  {$\Theta(\psi^{top})$};
     \node[] (F) at (6.5,-4)  {$CL^{t}(\mathcal{A}^{\Psi})$};
          \node[] (G) at (8.8,-4)  {$\vdash$};
     \path [->] (F) edge node[above, yshift=0cm] {(xii)} (G);
     \path [dashed,->] (Cr) edge node[left, yshift=0cm] {(ix)} (F);

     \path [->] (Atomsch) edge node[right, yshift=0cm] {(xi)} (F);
     \path [->] (Bnc) edge node[above, yshift=0cm] {(x)} (Atomsch);
     \path [dashed,->] (E) edge node[above, yshift=0cm] {(ix)} (F);

    	\path [dashed,->] (A) edge node[above, yshift=0cm] {(ii)} (B);
    	\path [->] (A) edge node[left] {(iii)} (An);
    	\path [->] (B) edge node[right] {(iii)} (Bn);
	    	\path [dashed,->] (Bn) edge node[above] {(iv)} (An);
		\path [->] (Bn) edge node[left] {(v)} (C);
	    	\path [->] (An) edge node[below] {(v)} (C);
	    	\path [->] (C) edge node[left] {(v)} (D);
	    	\path [->] (D) edge node[below, rotate=0] {(vi)} (E);
	    	\path [dashed,->] (B) edge node[above] {(v)} (Bc);
	    	\path [->] (Bc) edge node[right] {(vi)} (Bnc);
	    	\path [->] (Bc) edge node[left] {(iii)} (C);
	    	\path [->] (down4c) edge node[left] {(vii)} (E);
	    	\path [dashed,->] (A0) edge node[above,xshift=0cm,yshift=0cm, rotate=0] {(i)} (A);

	    	\path [<->] (Cr) edge node[above] {(viii)} (down4c);

	    	\path [->] (Bnc) edge node[right] {(iii)} (Cr);
	    	\path [->] (C) edge node[below] {(vi)} (down4c);
	    	\path [dashed,->] (Cr) edge node[right] {(ix)} (E);

\end{tikzpicture}\end{small}
\end{center}
\caption{ Clause set subsumption relationships for proof schemata. (i):{\small Def.~\ref{def:AccSch}}, (ii):{\small Def.~\ref{def:LazyEval}}, (iii): Def.~\ref{def:eval}, (iv):{\small Lem.~\ref{lem:pairing}}, (v): Reductive cut-elimination, (vi): {\small Def.~\ref{def:CTS}}, (vii): Corollary~\ref{cor:subTANCF}, (viii): Proposition~\ref{prop:clauseequal} (ix): Theorem~\ref{thm:schclausalSub} and transitivity of $\leq_{ss}$, (x): Lemma~\ref{lem:comass}, (xi): Definition~\ref{def:topclass}, (xii) When there exists a substitution schema for $CL^{t}(\mathcal{A}^{\Psi})$. Note that {\em dashed arrows} represent our results and {\em solid arrows} existing results.} 
\label{fig:one}
 \end{figure}
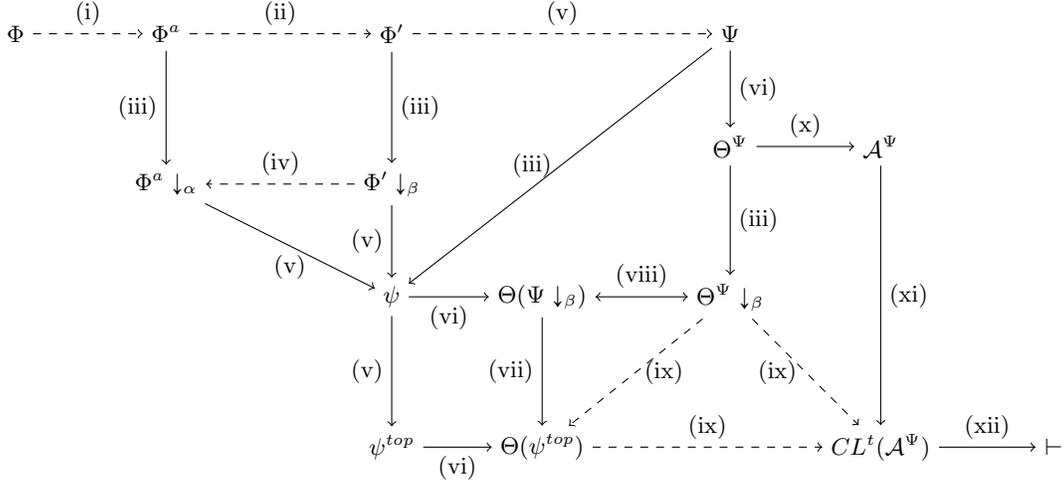

\section{Conclusion}
\label{sec:Conclusion}
In this work we extend results concerning the clausal analysis of \LK-proofs to  first-order proof schemata~\cite{BAAZ2006a}. The results of~\cite{BAAZ2006a}  are based on clause set subsumption  properties between an \LK-proof $\varphi$ and an \LK-proof  $\psi$ which is the result of applying reductive cut-elimination to $\varphi$, without atomic cut elimination. These results are not easily extendable to proof schemata because cuts cannot be easily passed through proofs links. We define lazy instantiation allowing   proof links to unfold and  allowing reductive cut-elimination to continue. Thus, we were able to derive clause set subsumption results for arbitrary applications of reductive cut-elimination to proof schemata. Furthermore, using the structural results of~\cite{ConduThesis,WolfThesis}, we were able to strengthen these results and reduce the problem of constructing an \ACNF\ schema (see~\cite{CERESS}) to find a substitution schema~\cite{CERESS}. As discussed in~\cite{MyThesis,DBLP:conf/cade/CernaL16}, defining the structure of the resolution refutation is the hardest part of the proof analysis using schematic \CERES~\cite{CERESS}. For future work, we plan to investigate how the introduced concepts can be applied to other issues concerning proof schemata. Also, we plan to investigate ways one can simplify the construction of substitution schemata, possibly through the use of automated theorem proving methods~\cite{DBLP:conf/cade/CernaL16}.

\section*{Acknowledgment}
We would like to thank Alexander Leitsch for reviewing and providing suggestions.

\bibliographystyle{plain}
\bibliography{Refs}
\end{document}

%% file: Symbols.tex
\usepackage{scalerel}
\newcommand*{\shifttext}[2]{%
  \settowidth{\@tempdima}{#2}%
  \makebox[\@tempdima]{\hspace*{#1}#2}%
}
\newcommand\fat[1]{\ThisStyle{\ooalign{%
  \kern.46pt$\SavedStyle#1$\cr\kern.33pt$\SavedStyle#1$\cr%
  \kern.2pt$\SavedStyle#1$\cr$\SavedStyle#1$}}}

\makeatother
\newcommand{\ES}[1]{\ensuremath{\textit{es}(#1)}}
\newcommand{\lyus}[1]{
	\if#10 \ensuremath{\mbox{\includegraphics[scale=.025]{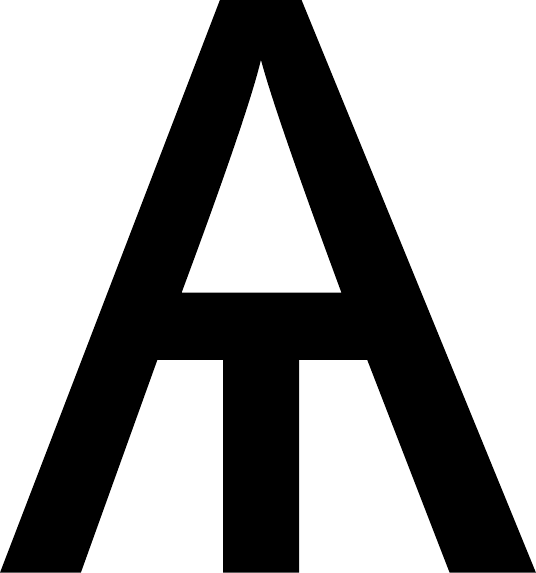}}}
	 \else
	  \if#11 \ensuremath{\mbox{\includegraphics[scale=.035]{yus}}}
	   \else
	    \if#12 \ensuremath{\mbox{\includegraphics[scale=.045]{yus}}}
	     \else
	      \if#13 \ensuremath{\mbox{\includegraphics[scale=.055]{yus}}}
	       \else
	        \if#14 \ensuremath{\mbox{\includegraphics[scale=.065]{yus}}}
	         \else
	          \if#15 \ensuremath{\mbox{\includegraphics[scale=.075]{yus}}}
	          \else \ensuremath{\mbox{\includegraphics[scale=.045]{yus}}}
	          \fi 
	         \fi 
	        \fi 
	       \fi 
	      \fi 
	     \fi 
	     \makebox[.05cm][l]{}
}

\newcommand{\byus}[1]{
	\if#10 \ensuremath{\mbox{\includegraphics[scale=.025]{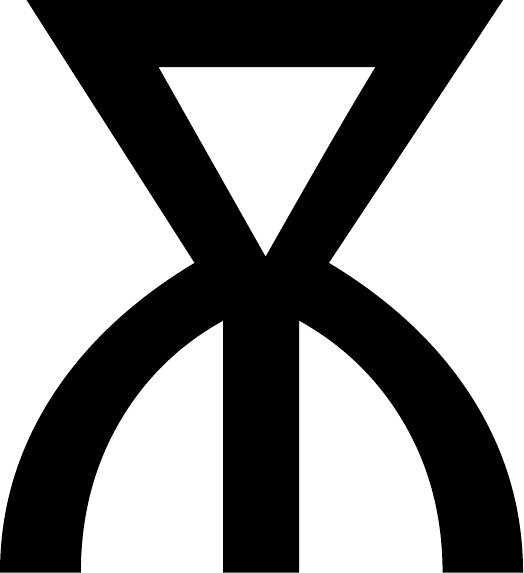}}}
	 \else
	  \if#11 \ensuremath{\mbox{\includegraphics[scale=.035]{byus}}}
	   \else
	    \if#12 \ensuremath{\mbox{\includegraphics[scale=.045]{byus}}}
	     \else
	      \if#13 \ensuremath{\mbox{\includegraphics[scale=.055]{byus}}}
	       \else
	        \if#14 \ensuremath{\mbox{\includegraphics[scale=.065]{byus}}}
	         \else
	          \if#15 \ensuremath{\mbox{\includegraphics[scale=.075]{byus}}}
	          \else \ensuremath{\mbox{\includegraphics[scale=.045]{byus}}}
	          \fi 
	         \fi 
	        \fi 
	       \fi 
	      \fi 
	     \fi 
\makebox[.05cm][l]{}
}

\newcommand{\ibyus}[1]{
	\if#10 \ensuremath{\mbox{\includegraphics[scale=.025]{ibyus}}}
	 \else
	  \if#11 \ensuremath{\mbox{\includegraphics[scale=.035]{ibyus}}}
	   \else
	    \if#12 \ensuremath{\mbox{\includegraphics[scale=.045]{ibyus}}}
	     \else
	      \if#13 \ensuremath{\mbox{\includegraphics[scale=.055]{ibyus}}}
	       \else
	        \if#14 \ensuremath{\mbox{\includegraphics[scale=.065]{ibyus}}}
	         \else
	          \if#15 \ensuremath{\mbox{\includegraphics[scale=.075]{ibyus}}}
	          \else \ensuremath{\mbox{\includegraphics[scale=.045]{ibyus}}}
	          \fi 
	         \fi 
	        \fi 
	       \fi 
	      \fi 
	     \fi 
\makebox[.05cm][l]{}
}

\newcommand{\byusa}[1]{
	\if#10 \ensuremath{\mbox{\includegraphics[scale=.025]{byusarrow}}}
	 \else
	  \if#11 \ensuremath{\mbox{\includegraphics[scale=.035]{byusarrow}}}
	   \else
	    \if#12 \ensuremath{\mbox{\includegraphics[scale=.045]{byusarrow}}}
	     \else
	      \if#13 \ensuremath{\mbox{\includegraphics[scale=.055]{byusarrow}}}
	       \else
	        \if#14 \ensuremath{\mbox{\includegraphics[scale=.065]{byusarrow}}}
	         \else
	          \if#15 \ensuremath{\mbox{\includegraphics[scale=.075]{byusarrow}}}
	          \else \ensuremath{\mbox{\includegraphics[scale=.045]{byusarrow}}}
	          \fi 
	         \fi 
	        \fi 
	       \fi 
	      \fi 
	     \fi 
}

\newcommand{\qmarrow}[1]{
	\if#10 \ensuremath{\mbox{\includegraphics[scale=.065]{qarrow}}}
	 \else
	  \if#11 \ensuremath{\mbox{\includegraphics[scale=.075]{qarrow}}}
	   \else
	    \if#12 \ensuremath{\mbox{\includegraphics[scale=.085]{qarrow}}}
	     \else
	      \if#13 \ensuremath{\mbox{\includegraphics[scale=.095]{qarrow}}}
	       \else
	        \if#14 \ensuremath{\mbox{\includegraphics[scale=.105]{qarrow}}}
	         \else
	          \if#15 \ensuremath{\mbox{\includegraphics[scale=.075]{qarrow}}}
	          \else \ensuremath{\mbox{\includegraphics[scale=.045]{qarrow}}}
	          \fi 
	         \fi 
	        \fi 
	       \fi 
	      \fi 
	     \fi 
}

\newtheorem{definition}{Definition}
\newtheorem{theorem}{Theorem}
\newtheorem{example}{Example}
\newtheorem{lemma}{Lemma}
\newtheorem{proposition}{Proposition}
\newtheorem{corollary}{Corollary}
\newtheorem{remark}{Remark}

\newcommand{\ACNF}{\textbf{ACNF}}
\newcommand{\ACNFtop}{\textbf{ACNF}$^{\text{top}}$}
\newcommand{\CERES}{\textbf{CERES}}
\newcommand{\LK}{\textbf{LK}}
\newcommand{\LKE}{\textbf{LKE}}
\newcommand{\LKS}{\textbf{LKS}}
\newcommand{\SU}{\leq_{ss}}
\newcommand{\cL}{\textit{cl}}